\documentclass{sn-jnl}
\usepackage{amsrefs}
\usepackage[all]{xy}
\usepackage{tikz}
\usepackage{tikz-cd}
\usetikzlibrary{arrows,matrix}
\usepackage{mathtools}
\usepackage{amssymb}
\usepackage{amsthm}
\usepackage{faktor}
\usepackage{lmodern}

\theoremstyle{plain}
\newtheorem{theorem}{Theorem}[section]
\newtheorem{lemma}[theorem]{Lemma}
\newtheorem{corollary}[theorem]{Corollary}
\newtheorem{proposition}[theorem]{Proposition}

\theoremstyle{definition}
\newtheorem{definition}[theorem]{Definition}
\newtheorem{example}[theorem]{Example}

\theoremstyle{remark}

\newcommand{\textetaleaxiom}[1]{\text{$(\acute{E}t)_{#1}$} }
\newcommand{\etalevariety}[1]{\category{\acute{E}t}(#1) }

\newcommand{\mathimplies}{\Rightarrow}

\newcommand{\category}[1]{\mathbf{#1}}
\newcommand{\functor}[1]{#1}

\newcommand{\alg}[1]{\category{Alg}(#1)}
\newcommand{\powerset}[1]{\mathcal{P}(#1)}
\newcommand\restr[2]{#1\vert_{#2}}
\newcommand{\leftadjunct}[1]{#1_{*}}
\newcommand{\rightadjunct}[1]{#1^{*}}
\newcommand\upsetarrow {{\uparrow}}
\newcommand\downsetarrow {{\downarrow}}

\newcommand{\esakiadualhomo}[1]{\widehat{#1}}

\begin{document}

\title[\'{E}tale algebras over finite Heyting algebras]{\'{E}tale algebras over finite Heyting algebras}

\author{
    \fnm{Evgeny} 
    \sur{Kuznetsov}
    }
\email{e.kuznetsov@freeuni.edu.ge}

\affil{
\orgdiv{Department of Mathematical Logic}, 
\orgname{Andrea Razmadze Mathematical Institute of I. Javakhishvili Tbilisi State University}, 
\orgaddress{
            \street{2 Merab Aleksidze II Lane}, 
            \city{Tbilisi}, 
            \postcode{0193},  
            \country{Georgia}
            }
    }

\abstract{
In this paper, we investigate the concept of local homeomorphism in Esakia spaces. We introduce the notion of \'{e}tale Heyting $H$-algebra and establish category-theoretic duality for \'{e}tale Heyting $H$-algebra in the case of finite Heyting algebra $H$. Furthermore, we give an identity that axiomatizes the variety of \'{e}tale Heyting $H$-algebras when $H$ is finite. We also show that the category of Stone space-valued (co)presheaves over a finite Esakia space $X$ is equivalent to the slice category of local homeomorphisms over $X$. The fact is used to show that, in comparison with the case of general Heyting $H$-algebras, it is easier to compute finite colimits in the category of \'{e}tale Heyting $H$-algebras.
}

\keywords{
Heyting algebra,\sep Esakia duality,\sep \'{E}tale algebra,\sep Expansion by constants
}

\maketitle

\section{Introduction}\label{intro}

In the 20\textsuperscript{th} century, deep connections between algebra and geometry were extended to the realm of logic through algebraic and categorical logic. This perspective on logic sheds new light on familiar results while also yielding novel insights.

One of the convenient notions to employ when working with profound symmetry between algebra and geometry is the concept of equivalence of categories and dual equivalence, also known as dualities.

Despite its highly abstract nature, this symmetry proves remarkably fruitful, forging connections between seemingly disparate branches of mathematics. Notably, these category-theoretic dualities act as bridges between algebra and geometry, illuminating distinct conceptual facets of a subject. Among other notable dualities are, for example, the Gelfand-Naimark duality between $C^{*}$-algebras and the category of compact Hausdorff spaces \cite{GelNeu43} \cite[Ch. 4]{stone-spaces-johnstone}; Grothendieck duality for commutative rings and affine schemes \cite{HarsthorneResidues}, \cite[Ch. 5]{stone-spaces-johnstone}; extremely beautiful duality called Isbell duality \cite{adequate-isbell},\cite{avery-leinster-on-isbell}; Stone duality between Boolean algebras and compact Hausdorff zero-dimensional topological spaces \cite{davey_priestley_2002}; as well as Priestley \cite{priestley-1972} and Esakia \cite{esakia1974} dualities for distributive lattices and Heyting algebras, respectively.

This paper aims to study mappings between Esakia spaces that are locally invertible around every domain element, employing Esakia duality. Later, we will refer to such functions as local homeomorphisms. The concept will be further elaborated below, as defined in Definitions \ref{localHomeomorphismDefinition} and \ref{pMorphismDefinition}. Before delving into that, it's worth noting that Esakia spaces are objects closely linked through Esakia duality to Heyting algebras, which are algebraic models of intuitionistic propositional logic.

Specifically, we establish criteria for a map between Esakia spaces to be locally invertible around every point of its domain, and we investigate the algebraic duals associated with such objects. We will also demonstrate that the objects under consideration play a valuable role in studying Heyting algebras.

As a reference to basic notions from category theory, we use nicely written textbooks \cite{riehl-category-book}, \cite{leinster_book}, \cite{maclane-category-book}. We will use the notions from category theory explicitly, often even implicitly, to formalize other concepts we employ.

We mainly follow the book \cite{burrisAndSankappanavar} for concepts within the field of algebraic logic, including \emph{distributive lattice}, \emph{Heyting algebra}, \emph{Boolean algebra}, \emph{prime filter}, \emph{ultrafilter}, \emph{principal filter}, \emph{variety of universal algebras}, and more.

For notions from general topology, we mainly refer to, e.g., \cite{engelking} and \cite{willard}.

The article is intentionally written neutrally, often devoid of shortcut category-theoretic terminology, to be accessible to the general public.

\subsection{Stone duality and Esakia duality}
Let us begin by revisiting the topological and ordered-topological dualities for structures found in algebraic logic, such as bounded distributive lattices, Boolean algebras, and Heyting algebras.

The works of Marshall H. Stone on topological duality for distributive lattices and later for Boolean algebras have laid the groundwork for understanding the dual relationship between Boolean algebras and compact Hausdorff spaces \cites{stone_1_36,stone_2_38}, \cite{stone-spaces-johnstone}. Stone's representation theorem reveals the intricate relationship between the algebraic structure of Boolean algebras and the topology of specific spaces, offering a powerful tool for analyzing Boolean algebras through their dual spaces.

 We assume that the reader is familiar with the notions of \emph{compact} topological space; see, e.g., \cite[Definition 17.1]{willard}, \emph{Hausdorff} space; see, e.g., \cite[Definition 13.5]{willard}, and \emph{zero-dimensional} topological space; see, e.g., \cite[Ch. 6, \S2]{engelking}.

\begin{definition}
    A topological space $X$ is called \emph{Stone space} if it is compact, Hausdorff, and zero-dimensional.
\end{definition}

We use $\category{S}$ to denote the category of Stone spaces and continuous maps.

To comprehend the relationship between Boolean algebras and Stone spaces, one uses the notion of dual equivalence of categories; see, e.g., \cite[\textsection VI.4]{maclane-category-book}. Moreover, to define dual equivalence of categories, we will also need a notion of opposite category; see \cite[\textsection II.2]{maclane-category-book}. Using these two notions, we can briefly define the concept of dual equivalence (duality) between two categories.

\begin{definition}
    A \emph{dual equivalence} (\emph{dual equivalence}) between categories $\category{A}$ and $\category{B}$ is an equivalence between the opposite category $\category{A}^{op}$ and the category $\category{B}$.
\end{definition}

We now possess sufficient notions to explain Stone duality. In particular, M. H. Stone's work \cites{stone_1_36, stone_2_38} demonstrates that the category of Stone spaces and continuous functions is dually equivalent to the category of Boolean algebras and Boolean algebra homomorphisms.

At the level of objects, the Stone duality is as follows; see, e.g., \cite{davey_priestley_2002}: 
\begin{itemize}
\item For the Stone space $X$, the corresponding Bolean algebra $B_X$ is the algebra of clopen subsets of $X$.
\item On the other hand, for a Boolean algebra $B$, the underlying set of the corresponding Stone space $X_B$ consists of prime filters of $B$ (which, in the case of Boolean algebras, are the same as ultrafilters; see \cite[Theorem 3.12]{burrisAndSankappanavar}. The basis for topology on $X_B$ consists of subsets $C_a=\{U\mid a\in U\}$ of prime filters $U$ of $B$ for every element $a\in B$.
\item Stone duality sends continuous functions between Stone spaces to Boolean algebra homomorphisms between corresponding Boolean algebras, and vice versa.
\end{itemize}

In the beginning of the 1970s, H. A. Priestley provided order topological duality for distributive lattices \cites{priestley-1970, priestley-1972}. Relatively at the same time, L. Esakia explored order topological duality for Heyting algebras, providing a profound understanding of the interplay between topological structures and algebraic structures related to intuitionistic propositional logic \cites{esakiaBook, n_bezhanishvli_dissertation}.

To explore further, we will employ the concept of a partially ordered set (\emph{poset}); see e.g., \cite[Definition 1.2.]{burrisAndSankappanavar}.

A certain family of subsets conveniently emphasizes and shapes the geometric properties of partially ordered sets. The following definition describes the concept of upward-closed and downward-closed subsets within a partially ordered set.

\begin{definition}
    A subset $A$ of a partially ordered set $X$ is called \emph{upward-closed} (or \emph{upset}) if for arbitrary $x\in A$ and  arbitrary $x'\in X$, $x\leqslant x'$ implies $x'\in A$. A subset $A$ is called \emph{downward-closed} (or \emph{downset}) if for arbitrary $x\in X$,  $x\in A$ and $x'\leqslant x$ imply $x'\in A$. 
\end{definition}

Note that within upsets and downsets, some are generated by a single element. 

\begin{definition}\label{upset-definition} 
   For an element $x$ of a partially ordered set $X$, a \emph{principal upset} generated by $x$ is a subset ${\upsetarrow x=\{x'\mid x\leqslant x' \}}$; similarly, the subset ${\downsetarrow x =\{x'\mid x'\leqslant x \}}$ is called \emph{principal downset}.
\end{definition}
Principal upsets and downsets play a special role in this theory because $\upsetarrow x$ and $\downarrow x$ are the smallest upset and downset containing $x$, respectively.

Note also that the collection of all upsets $\category{Up}(X)$ (downsets $\category{Down}(X)$, respectively) on a partially ordered set $X$ is closed under arbitrary intersections and unions and contains $\emptyset$ as well as $X$. This observation tells us that the collection of upsets (alternatively downsets) forms an Alexandroff topology on $X$; see, e.g., \cite{Alexandroff-37}, \cite[Ch. II, \S 1.8]{stone-spaces-johnstone}. Later, we refer to this topology as the \emph{upset topology} (\emph{downset topology}, respectively) on $X$.

Using the terms that have been previously defined, we can now recall two classes of partially ordered topological spaces called \emph{Priestley spaces} and \emph{Esakia spaces}.

\begin{definition}\label{esakia_space_definition}
A \emph{Priestley space} is a partially ordered topological space that is compact, satisfying the condition that for all $x\not\leqslant y$, there exists a clopen upset $U$ with $x \in U$ and $y\not \in U$. The last condition is called \emph{Priestley separation axiom}. A Priestley space is called \emph{Esakia space} if, moreover, for each clopen set $C$, the smallest downset $\downsetarrow C$ containing $C$ is also clopen.
\end{definition}
Note that there are also other equivalent definitions of Esakia spaces, some of which avoid explicit use of the Priestley separation axiom; see, for example, \cite{esakiaBook}.

The following two lemmas, in particular, state some useful properties of Priestley spaces and Esakia spaces. For the proof, please consult \cite{priestley-1972}.

\begin{lemma}\label{esakia_space_pointed_upset_is_closed}
    In a Priestley space $X$ the subsets $\upsetarrow x$ and $\downsetarrow x$ are closed subsets for each $x\in X$. 
\end{lemma}

Moreover, along with compactness, the Priestley separation axiom implies much stronger topological properties for a Priestley space; see, e.g., \cite{davey_priestley_2002} for proof.

\begin{lemma}
Every Priestley space is Hausdorff and zero-dimensional. By definition, it is compact; hence, every Priestley space is a partially ordered Stone space.
\end{lemma}

Our primary focus is on studying Priestley spaces and Esakia spaces within the framework of category theory. However, to define the categories of Priestley spaces and Esakia spaces, it is necessary to also specify the morphisms between these objects.

To avoid ambiguity, we state the following straightforward definition:

\begin{definition}\label{monotone-function-definition}
A map $f\colon X'\to X$ between partially ordered sets $(X', \leqslant')$ and $(X, \leqslant)$ is referred to as \emph{order-preserving} if $x_1 \leqslant' x_2$ implies $f(x_1) \leqslant f(x_2)$.   
\end{definition}

Morphisms in the category of Priestley spaces are the continuous order-preserving maps between Priestley spaces.

\begin{definition}\label{pMorphismDefinition}
A map between partially ordered sets $X$ and $Y$ is called a \emph{$p$-morphism} if it is order-preserving, and for any $y \geq f(x)$ with $x\in X$ and $y\in Y$, there is an $x' \geq x$ with $f(x')=y$. We refer to the last condition of a $p$-morphism as the `\emph{back}' condition. A $p$-morphism is called \emph{strict} if the $x'$ in the `back' condition is moreover unique.    
\end{definition}

A standard Esakia morphism $f\colon X'\to X$ in the category $\category{E}$ of Esakia spaces is a function that is continuous with respect to the underlying Stone topologies, which additionally is a $p$-morphism with respect to the corresponding partial order relations on $X'$ and $X$. Note that we have used the word `standard' above. This is because, in what follows, we are going to restrict our attention to those Esakia morphisms that are strict $p$-morphisms.

To the best of the author's knowledge, the notion of $p$-morphism and the simplest form of the representation of Heyting algebras by partially ordered sets were introduced in \cite{troelstra1966}. However, the complete form of order-topological duality for Heyting algebras was first introduced by Leo Esakia in \cite{esakia1974}; see also \cite{esakiaBook}. In essence, Esakia duality is an extension of the Stone duality for Boolean algebras and, in particular, of the Priestley duality for distributive lattices \cites{priestley-1970, priestley-1972}. Note also that, relatively at the same time, the analogous duality appears in \cite{adams1984}.

\begin{itemize}
 \item At the level of objects, Priestley duality assigns an ordered topological space, denoted as $X_L$, to a bounded distributive lattice $L$. The underlying set of points in $X_L$ comprises prime filters of $L$, equipped with a topology generated by subsets $C_a=\{U, X\setminus U\mid U \text{ is a prime filter of } L \text{ and } h\in U\}$ for every $a\in L$. The partial order on $X_L$ is defined as $U\leqslant V$ if and only if $U\subseteq V$. If $L$ is also isomorphic to a Heyting algebra $H$, the Priestley space $X_H$ turns out to be an Esakia space.

\item On the other hand, to a Priestley space $X$ corresponds the bounded distributive lattice $L_X$ of \emph{clopen upsets} of $X$. If $X$ is also an Esakia space, then $L_X$ is a Heyting algebra.

\item Priestley duality maps continuous order-preserving functions between Priestley spaces to homomorphisms between corresponding bounded distributive lattices, and vice versa. Similarly, under Esakia duality, continuous $p$-morphisms between Esakia spaces are mapped to Heyting algebra homomorphisms between corresponding Heyting algebras, and vice versa.
\end{itemize}

The underlying set of an Esakia space $X$ is finite if and only if $H_X$, the Esakia dual Heyting algebra of $X$, is finite; see, e.g., \cite[Theorem 3.4.17]{esakiaBook}.

The Esakia duality theorem posits that the correspondence between the category $\category{E}$ of Esakia spaces and standard Esakia morphisms and the category $\category{H}$ of Heyting algebras and Heyting algebra homomorphisms, as described above, extends to a pair of functors that establish a dual equivalence of categories; see, e.g., \cite{esakia1974}, and also \cites{n_bezhanishvli_dissertation}. Priestley duality similarly establishes a dual equivalence between the category of Priestley spaces and continuous order-preserving functions and the category of bounded distributive lattices and lattice homomorphisms; see, e.g.,\cites{priestley-1970, priestley-1972} and \cite{davey_priestley_2002}.

\section{Local homeomorphisms}
This paper's major goal is to investigate the local invertibility of Esakia morphisms around each element of the domain (often referred to as \emph{local homeomorphisms}; see Definition \ref{localHomeomorphismDefinition} below). We then investigate possible algebraic characterizations linked to them via Esakia duality.

Note also that local homeomorhisms of topological spaces can be described in terms of sheaves \cite{maclane-moerdijk-sheaves}, a machinery widely used throughout mathematics for almost a century. Sheaves originated in mathematical analysis and geometry, where they serve as cohomology coefficients and representations of functions suitable for different types of manifolds. On the other hand, logic uses sheaves as vehicles for models of set theory.

We hope that the study of local homeomorphisms in Esakia spaces contributes to a better understanding of Heyting algebras and aligns with broader trends in the area of research.

\subsection{Local homeomorphisms of topological spaces}

Let us begin our advancement with the well-studied example of topological spaces.

For notational convenience, for a map $f\colon X'\to X$, we consider two functions between corresponding power sets, namely ${\leftadjunct{f}\colon \powerset{X'}\to \powerset{X}}$ and ${\rightadjunct{f}\colon \powerset{X}\to \powerset{X'}}$ defined as follows: ${\leftadjunct{f}(U)=\{f(x)\mid x\in U\}} $ and ${\rightadjunct{f}(V)=\{x\mid f(x)\in V\}}$.

These two functions, $\leftadjunct{f}$ and $\rightadjunct{f}$\!, preserve the subset inclusion order on powersets. If they are seen as functors between poset categories (see the Definition \ref{poset-category} below), they make an obvious adjunction, meaning $U \subseteq \rightadjunct{f}(V)$ if and only if $\leftadjunct{f}(U) \subseteq V$.

We assume the reader is familiar with the notions of \emph{open map}, \emph{closed map}, and \emph{homeomorphism}. For details related to these notions, please refer to, e.g., \cite[\S 1.4]{engelking}.

 \begin{definition}\label{localHomeomorphismDefinition}
A map $f\colon X'\to X$ between topological spaces $X'$ and $X$ is said to be \emph{local homeomorphism} (or in other terms \emph{locally invertible} around every element of its domain $X'$), if for every element $x\in X'$ there exists an open subset $U_x\subseteq X'$ such that the image $\leftadjunct{f}(U_x)$ is an open subset in $X$ and the restriction of $f$ to $U_x$ is a homeomorphism between $U_x$ and $\leftadjunct{f}(U_x)$ equipped with the subspace topology induced from $X'$ and $X$ respectively.
 \end{definition}

We also encounter the following terminology in the subject literature: a continuous map $f\colon X' \to X$ is referred to as a \emph{space over} $X$ (or sometimes simply \emph{bundle}), while a local homeomorphism $f\colon X' \to X$ is termed an \emph{\'{e}tale space} over $X$ (or \emph{\'{e}tale map} over $X$). This terminology justifies both the title of the article and the name of the algebraic construction of \'{E}tale Heyting algebra that will be discussed below.

One example is the category $\category{LH}/X$, which consists of all the local homeomorphisms over the fixed topological space $X$. This is one way to describe the category of sheaves of continuous maps over $X$; see, for example, \cite[Ch. 2, \S 6]{maclane-moerdijk-sheaves}.

The sequence of well-known results that follows demonstrates the formation of a locally injective continuous map from locally invertible pieces. We can view this instance of local homeomorphisms in topological spaces as a simple example that illustrates more intricate and fruitful instances of a powerful and involved idea.

We omit the proof for the following useful lemma because it is almost trivial.
\begin{lemma}\label{bijection-lemma}
If $f\colon X'\to X$ is a bijection then for every subset $S\subseteq X$ holds \[\leftadjunct{f}(X'\setminus S) = X\setminus \leftadjunct{f}(S).\]
\end{lemma}

Using this last observation, the following proposition can be demonstrated with ease; for proof, see, e.g., \cite[Proposition 1.4.18]{engelking}.

\begin{corollary}\label{continuous-inverse-lemma}
    For a bijection $f\colon X'\to X$ the following are equivalent
    \begin{itemize}
        \item $f^{-1}\colon X\to X'$ is continuous;
        \item $f$ is open;
        \item $f$ is closed;
    \end{itemize}
\end{corollary}

Nevertheless, in some contexts, the continuous function is a local homeomorphism with less assumption. We only recall the well-known case when the domain of the continuous function is compact and the codomain is Hausdorff. For proof of the following two statements, refer to, e.g., \cite[Theorem 3.1.12 and Theorem 3.1.13]{engelking}.

\begin{lemma}\label{closed-function-lemma}
    Every continuous map $f\colon X' \to X$ from
    compact space $X'$ to Hausdorff space $X$ is closed.
\end{lemma}

\begin{corollary}
  A continuous bijection from a compact space to a Hausdorff space is a homeomorphism.  
\end{corollary}

To facilitate future development, it is important to note that we can also regard functions as objects within a certain category, as outlined in the following definition; see, e.g., \cite[\S II.6]{maclane-category-book}.

\begin{definition}
    If $\category{C}$ is a category and $X$ is an object in $\category{C}$ then 
    \begin{itemize}
        \item    the \emph{category of objects under} $X$, denoted as $X/\category{C}$, is defined as a category where an objects are arrows $f\colon X\to A$ in $\category{C}$ with domain $X$. An arrow $h$ in $X/\category{C}$, from an object $f\colon X\to A$ to an object $g\colon X\to B$, is represented by an arrow $h\colon A\to B$ in $\category{C}$, forming a commutative the triangle $g=h\circ f$.
        \item Similarly the \emph{category of objects over} $X$, denoted as $\category{C}/X$, is  a category where objects are arrows $f\colon A\to X$ in $\category{C}$ with co-domain $X$. An arrow $h$ in $\category{C}/X$ is given by an arrow $h\colon A\to B$ in $\category{C}$ making the triangle $f=g\circ h$ commutative.
    \end{itemize}       
\end{definition}

\subsection{Ordered and ordered topological analogues}

The objective of the present study arose in search of a new class of morphisms of Esakia spaces called \emph{Esakia local homeomorphisms}, providing some evidence that for an Esakia space $X$, the corresponding category $\category{LH}_{\category{E}}/X$ of Esakia local homeomorphisms over $X$ enjoys many useful properties.

In this paper, the role of local homeomorphisms over Esakia spaces is played by continuous strict $p$-morphisms. We will explain and elaborate on the rationale for doing so. In the sequel, $\category{LH}_{\category{E}}/X$ denotes the category with objects, strict $p$-morphisms over a fixed Esakia space $X$ in the category of Esakia spaces. The arrows in $\category{LH}_{\category{E}}/X$ are standard Esakia morphisms between objects' domains, forming commutative triangles.

It is well known that if $f\colon Y\to X$ and $f'\colon Y'\to X$ are local homeomorphisms of topological spaces and $g\colon Y\to Y'$ is a continuous map such that $f'\circ g=f$, then $g$ is a local homeomorphism. Similarly, if $f\colon Y\to X$ and $f'\colon Y'\to X$ are strict $p$-morphisms, and $f'\circ g=f$ for the continuous order-preserving map $g\colon Y\to Y'$, then $g$ is also a strict $p$-morphism.

As we will see below, the categories $\category{LH}_{\category{E}}/X_{H}$ provide useful tools for studying general Heyting algebras, where $X_H$ is the dual Esakia space for a Heyting algebra $H$. In particular, if $H$ is a finite Heyting algebra, then finite limits in the category $\category{LH}_{\category{E}}/X_{H}$ are well-behaved and can be explicitly described.

On the Heyting algebra side, this provides a class of Heyting algebra homomorphisms with manageable amalgamation properties; see, e.g., the Corollary \ref{finite-colimits-cor}. This is important because pushouts of general Heyting algebra homomorphisms are notoriously difficult to describe.

The following well-known lemma characterizes order-preserving maps in terms of preservation upward-closedness by $\leftadjunct{f}$; see, e.g., \cite[Lemma 1.4.12]{esakiaBook} for proof.

\begin{lemma}
Let $X'$ and $X$ be partially ordered sets and $f\colon X'\to X$ be a map, then

\begin{enumerate}
\item $f$ is order-preserving if and only if $\leftadjunct{f}(\upsetarrow x)\subseteq \upsetarrow f(x)$.
\item $f$ satisfies the `back' condition of a $p$-morpism  if and only if $\upsetarrow f(x)\subseteq \leftadjunct{f}(\upsetarrow x)$.

\end{enumerate}
\end{lemma}

We also use a well-known fact relating the properties of being an order-preserving map, a continuous map, a $p$-morphism, and an open map; see, e.g., \cite[Lemma 3.3.2]{esakiaBook}.

\begin{lemma}\label{contunuity-and-openness-is-monotonicity-and-pmorhismness}
    Let $X'$ and $X$ be partially ordered sets and $f\colon X'\to X$ be a map, then
    \begin{enumerate}
        \item $f$ is continuous with respect to upsets topologies on $X'$ and $X$ if and only if it is order-preserving;
        \item $f$ is open with respect to upsets topologies on $X'$ and $X$ if and only if the `back' condition of a $p$-morpism holds.
    \end{enumerate}
\end{lemma}

\begin{lemma}
Let $f\colon (X', \tau', \leqslant')\to (X, \tau,\leqslant)$ be a map between partially ordered topological spaces. Let for any $x\in X'$ there exist an open upset neighborhood $U_x$ of $x$ such that $\leftadjunct{f}(U_x)$ is an open upset and the restriction $\restr{f}{U_x}$ is a homeomorphism with respect to both ordinary and upset topologies. Then $f$ is an open, continuous, strict $p$-morphism.
\end{lemma}
\begin{proof}
First, let us show that $f$ is continuous. Let $V$ be an open subset of $X$ and $x\in f^{*}(V)$. By assumption, there exists an open neighborhood $U_x$ of $x$ such that $\restr{f}{U_x}$ is a homeomorphism between $U_x$ and $\leftadjunct{f}(U_x)$.

Subsets $V$ and $\leftadjunct{f}(U_x)$ are open subsets in $X$, so $V\cap \leftadjunct{f}(U_x)$ is an open subset in $X$. $\restr{f}{U_x}$ is a homeomorphism, and as a result is continuous, then $\rightadjunct{\restr{f}{U_x}}(V\cap \leftadjunct{f}(U_x))=\{x\in U_x\mid f(x)\in V\cap \leftadjunct{f}(U_x)\}=\left(\rightadjunct{f}(V)\cap U_x\right) $ is open in $X$ and $x\in \left(\rightadjunct{f}(V)\cap U_x\right) \subseteq \rightadjunct{f}(V)$. Hence, every $x\in \rightadjunct{f}(V)$ has an open neighborhood contained in $\rightadjunct{f}(V)$, and $\rightadjunct{f}(V)$ is open. 

Similarly, we can show that $f$ is order-preserving. Namely, using Lemma \ref{contunuity-and-openness-is-monotonicity-and-pmorhismness}, it is easy to see that a map is order-preserving if and only if for any upset $V$ in $X$ the subset $\rightadjunct{f}(V)$ is an upset in $X'$.

Let us now prove that $f$ is open. Let $A$ denote an open subset of $X'$. For each $x\in A$, choose $U_x\subseteq X$ and $V_x\subseteq X'$ so that $\restr{f}{U_x}$ is a homeomorphism between $U_x$ and $V_x$. For each $x\in A$, the subset $\leftadjunct{f}(U_x\cap A)$ is an open subset in $V_x$, so it is open in $X$ as well. Therefore, $\bigcup_{x\in A}\leftadjunct{f}(U_x\cap A)=\leftadjunct{f}(A)$ is an open subset of $X$. 

In the same way, we can demonstrate that $f$ is a $p$-morphism since an order-preserving map is a $p$-morphism if and only if the subset $\leftadjunct{f}(A)$ is an upset in $X$ for each upset $A$ in $X'$. Lemma \ref{contunuity-and-openness-is-monotonicity-and-pmorhismness} is used again to see this.

Now let us show that $f$ is strict. Let $x\in X'$ and $f(x)\leqslant x'$. Since $f$ is a $p$-morphism, there exists $y\in \upsetarrow x$ such that $f(y)=x'$. Because there exist open upset neighborhoods $U_x$ and $V_{f(x)}$ such that $\restr{f}{U_x}$ is a homeomorphism between them (with respect to upset topologies) and $\upsetarrow x\subseteq U_x$, it follows that such $y\in \upsetarrow x$ is unique.
\end{proof}

Recall that according to Corollary \ref{continuous-inverse-lemma}, for a continuous mapping $f\colon X'$ to be a local homeomorphism, for every point $x\in X'$ there must exist an open neighborhood $U_x$ of $x$ such that $\leftadjunct{f}(U_x)$ is an open subset and the restriction map $\restr{f}{U_{x}}$ is a closed/open bijection.
Also, by Lemma \ref{esakia_space_pointed_upset_is_closed}, in an Esakia space (and in a Priestley space in general, see \cite{priestley-1972}), the upset $\upsetarrow x$ is always a closed subset.

With this fact in mind, if $f$ is a strict $p$-morphism between Esakia spaces, we can relax the assumptions on $f$ by dropping the openness and closeness restrictions without losing the property of being a local homeomorphism of underlying topological spaces.

\begin{lemma}
Suppose $f\colon X'\to X$ is a continuous strict $p$-morphism between Esakia spaces $X'$ and $X$. Then $\restr{f}{\upsetarrow x}$ is a continuous bijective $p$-morphism between $\upsetarrow x$ and $\leftadjunct{f}(\upsetarrow x)$, with inverse $\restr{f}{\upsetarrow x}^{-1}$ a continuous $p$-morphism.
\end{lemma}
\begin{proof}
Indeed, because $f$ is strict, it establishes a bijection between $\upsetarrow x$ and $\leftadjunct{f}(\upsetarrow x)$. Because $f$ is a $p$-morphism, $\leftadjunct{f}(\upsetarrow x)=\upsetarrow f(x)$. By Lemma \ref{esakia_space_pointed_upset_is_closed} $\upsetarrow x$ is a closed subset of a compact space; hence, $\upsetarrow x$ is compact. Now, $\leftadjunct{f}(\upsetarrow x)$ is a subspace of a Hausdorff space; hence, it is a Hausdorff space. By leveraging Lemma \ref{closed-function-lemma}, along with Corollary \ref{continuous-inverse-lemma} and assuming that $f$ is a continuous $p$-morphism, we can see that the inverse $f^{-1}$ is also a continuous $p$-morphism.
\end{proof}

The \emph{local sections} of a function $f\colon X'\to X$ over a subset $U\subseteq X$ are functions $s\colon U\to X'$ such that $f\circ s=id_U$. It's important to note that in more standard situations, like, e.g., topological spaces, the images $\leftadjunct{s}(U)$ of continuous local sections $s$ for \'{e}tale function $f$ form a base for the topology of the total space $X'$; see, e.g., \cite[Ch. 2, \S 6]{maclane-moerdijk-sheaves}. In other words, every member of the Heyting algebra of open subsets of $X'$ is a union of images $\leftadjunct{s}(U)$.

In our case, we have demonstrated that continuous strict $p$-morphisms are invertible on principal upsets $\upsetarrow x$. Because $f$ is a $p$-morphism, we also know that $\leftadjunct{f}(\upsetarrow x)=\upsetarrow f(x)$. For arbitrary $x \in X'$, the function $s = \restr{f}{\upsetarrow x}^{-1}$, which is the inverse of $\restr{f}{\upsetarrow x} \colon \upsetarrow x \to \leftadjunct{f}(\upsetarrow x)$, serves as a local section for $f$. In the case when $X$ is a finite Esakia space, principal upsets are clopen subsets because the Stone topology on a finite Esakia space is discrete. And finally, if $X$ is an Esakia space and $C$ is a clopen upset in $X$, then it is evident that $C = \bigcup_{c \in C} \upsetarrow c$. This demonstrates that every element in the Heyting algebra of clopen upsets is a union of images $\leftadjunct{s}(\upsetarrow f(x))$ of local sections of $f$. This is a crucial fact to be aware of.

\section{\'{E}tale Heyting algebras}
In this section, we introduce some notions from universal algebra serving algebraic duals to local isomorphisms of Esakia spaces. The analogues of these notions appear in various fields of mathematics and are studied in different contexts. Structures under consideration in this section can be thought of as an expansion of algebraic structure by a substructure of constants. It is worth noting that a concept seemingly closely related yet distinct in appearance, known as the $\forall$-subalgebra, has been recently considered by \cite[Definition 3.3]{almeida-2023-pi2rule}.

In what follows, we will utilize the notion of the \emph{variety} of algebras of some fixed signature. Readers not familiar with the notion of variety can refer to \cite[Ch. 2, Definition 9.3]{burrisAndSankappanavar} for clarification.

Let $\mathcal{V}$ be a variety of universal algebras, and let $A_0\in \mathcal{V}$ be an algebra in it. We consider the category $\alg{A_0}$ of $A_0$-algebras.

Namely, objects of $\alg{A_0}$ are $\mathcal{V}$-homomorphisms $A_0\to A$, and a morphism from $\alpha\colon A_0\to A$ to $\alpha'\colon A_0\to A'$ is represented by a $\mathcal{V}$-homomorphism $\beta\colon A\to A'$ such that $\beta \circ \alpha = \alpha'$.

Upon closer inspection, a reader may recognize the idea of expansion of the signature by subalgebra of constants here. The idea of expansion of the signature by constants is not novel and is used fruitfully in different contexts in mathematics, e.g., see \cite{taylor2017}, \cite{nouri2020}, \cite[\textsection 3]{daniyarova2012}. Note that indeed $\alg{A_0}$ forms the category of all algebras in another variety with an enriched signature, which we will also denote by $\alg{A_0}$. Operations of $\alg{A_0}$ are operations of $\mathcal{V}$ together with constants (nullary operations) $c_a$, one for each $a\in A_0$, and identities of the variety $\alg{A_0}$ are given by identities of $\mathcal{V}$ together with variable-free identities of the form $w(c_{a_1},\ldots,c_{a_n })=c_{w(a_1,\ldots,a_n)}$ for every generating $n$-ary operation $w$ of $\mathcal{V}$.

Let finally $\etalevariety{A_0}$ be the subvariety of $\alg{A_0}$ generated by $A_0$, i.e., by the identity map of $A_0$ viewed as the \emph{initial object} of the category $\alg{A_0}$. For the definition of the \emph{variety generated} by a class of algebras, please refer to \cite[Ch. 2, Definition 9.4]{burrisAndSankappanavar}. Also, recall that an object $A$ of a category $\category{C}$ is called \emph{initial} if there is exactly one morphism from $A$ to any object $B$ of $\category{C}$.

We refer to algebras from $\etalevariety{A_0}$ as \emph{\'{e}tale $A_0$-algebras}, and the $\mathcal{V}$-homomorphisms $A_0\to A$ that define \'{e}tale $A_0$-algebras as \emph{\'{e}tale homomorphisms}. As we will see, the name \'{e}tale will be subsequently justified for finite Heyting $H$-algebras due to their close relation with strict Esakia morphisms. In this context, the category of \'{e}tale $A_0$-algebras is the smallest subcategory of $\alg{A_0}$ containing $A_0$ and closed under products, subalgebras, and homomorphic images. According to Tarski's theorem \cite{tarski1946}, \cite[Ch. 2, Theorem 9.5]{burrisAndSankappanavar}, $\etalevariety{A_0}=\operatorname{HSP}(A_0)$ and by Birkhoff's theorem \cite{birkhoff1935} it is an equational class, see also \cite[Ch. 2, Theorem 11.9]{burrisAndSankappanavar}.

For Heyting algebra $H$, the corresponding category $\etalevariety{H}$ of \'{e}tale $H$-algebras, i.e., Heyting algebras $A$ equipped with an \'{e}tale homomorphism $H\to A$, admits a nicer and more explicit description compared to the category of arbitrary $H$-algebras $\alg{H}$, i.e., those given by arbitrary Heyting algebra homomorphisms $H\to A$. For instance, pushouts in this category are much simpler to describe than pushouts of general Heyting algebras. They coincide with pushouts computed in the category of distributive lattices and lattice homomorphisms; see Corollary \ref{pushout-corrolary}.

\section{Equivalencies and dualities for \'{etale} algebras over a finite Heyting algebra}
Let's begin this section with some additional definitions and notations for later use.

\begin{definition}
 Let $\category{C}$ and $\category{D}$ be two categories. The category, denoted by $\category{C}^{\category{D}}$, whose objects are (covariant) functors $\functor{F}\colon \category{D}\to\category{C}$ from $\category{D}$ to $\category{C}$, and morphisms are natural transformations between them, is called the category of $\category{C}$-valued (co)\emph{presheaves} over the category $\category{D}$.     
\end{definition}

There also exists an important type of category in which the collection of morphisms between two objects is either empty or contains a single member.

\begin{definition}\label{poset-category}
    A category $\category{C}$ is called a \emph{poset category} if there is at most one morphism between two objects in $\category{C}$.
\end{definition}
 
 Every partially ordered set $X$ can be regarded as a poset-category, where the objects are the elements of $X$ and the morphisms are determined by the partial order relation $\leqslant$ on $X$.

The remaining part of the paper is devoted to investigating the relationship between four categories, namely the category $\category{LH}_{E}/X_H$ of \'{e}tale Esakia morphisms over a finite partially ordered set $X_H$; the category $\etalevariety{H}$ of \'{e}tale $H$-algebras; the category $\category{S}^{X_H}$ of Stone spaces-valued (co)pre-sheaves on the finite partially ordered set $X_H$; and the variety $\mathcal{V}_{\text{\textetaleaxiom{H}}}$ of $H$-algebras validating additional axiom \textetaleaxiom{H} defined as follows $\underset{h\in H}{\bigvee} \left(x\Leftrightarrow c_h\right) = 1$, where $1$ is the greatest element of a Heyting algebra $H$ (note that $x$ is a variable here, and for the $H$-algebra, validating the axiom means that it holds for arbitrary $x$ from the codomain of the $H$-algebra).

Recall that `$\Leftrightarrow$' is not a part of the Heyting algebra signature. For $a, b \in H$, the expression $a \Leftrightarrow b$ is just shorthand notation for $(a \Rightarrow b) \wedge (b \Rightarrow a)$.
Note also that the analogue of the axiom \textetaleaxiom{H} has been studied previously in the context of elementary topoi and is adopted from \cite{jib}.

The diagram in Figure \ref{fig:outline} illustrates the connections (functors) between the main objects studied in this paper. Arrows indicate the direction of the functors involved in subsequent development. Two-headed arrows represent pairs of functors, establishing an equivalence of categories.

\begin{figure}
    \centering
    \begin{tikzpicture}[
    ->,
    shorten >=1pt,
    auto,
    node distance=3.5cm
    ]

    \node[] (leftmost) [] {$\etalevariety{H}$};
    \node[] (second) [right of = leftmost] {$\mathcal{V}_{\text{\textetaleaxiom{H}}}$};
    \node[] (third) [right of = second] {$\category{LH}_{\category{E}}/X$};
    \node[] (fourth) [right of = third] {$\category{S}^{X}$};

    \draw [->, shorten <=3pt, shorten >=3pt] (leftmost) -- (second) node[midway, above] {{\textup{Prop. } \ref{etaleAxiomInEtalevariety}}};

    \draw [<->, shorten <=3pt, shorten >=3pt] (second) -- (third) node[midway, above] {{\textup{Cor. } \ref{Fin_strictEsakiasAreEquivalentToHalgsWithEtaleAxiom}}};

    \draw [<->, shorten <=3pt, shorten >=3pt] (third) -- (fourth) node[midway, above] {\textup{Prop. } \ref{equivalenceOfStricEsakiaAndStonePresheaves}};

    \draw [->, shorten <=3pt, shorten >=3pt, out=195,in=-15] (fourth) to node[midway, below] {\textup{Prop. } \ref{StonePresheavesOnSierpinskiToEtaleVariety}} (leftmost);        
    
    \end{tikzpicture}
    \caption{Map of main results}
    \label{fig:outline}
\end{figure}

Let us start with categories $\etalevariety{H}$ and $\mathcal{V}_{\text{\textetaleaxiom{H}}}$\!\!. It is relatively easy to see that, for a finite Heyting algebra $H$, every member of $\etalevariety{H}$ validates the identity \textetaleaxiom{H}\!\!.
\begin{proposition}\label{etaleAxiomInEtalevariety}
If a Heyting algebra $H$ is finite, then every member of the variety $\etalevariety{H}$ validates the identity \textetaleaxiom{H} given by $\underset{h\in H}{\bigvee} \left(x\Leftrightarrow c_h\right) = 1$.

\end{proposition}

\begin{proof}
First, note that the identity homomorphism $1_H\colon H\to H$ validates \textetaleaxiom{H}\!\!. Indeed, if we substitute $x$ by an arbitrary element $h\in H$, then at least one member of the disjunction $\underset{h\in H}{\bigvee} \left(x\Leftrightarrow c_h\right) = 1$, and in particular, $h\Leftrightarrow c_h$, is equal to $1$, because $h=c_h$ by definition. 

On the other hand, $\etalevariety{H}$ is generated by $1_H$, and hence all the algebras in $\etalevariety{H}$ validate the \textetaleaxiom{H}\!{\!.}
\end{proof}

Now, let us consider the categories $\category{LH}_{\category{E}}/X$ and $\mathcal{V}{\text{\textetaleaxiom{H}}}$\!{\!.} Proposition \ref{strictEsakiasAreEquivalentToHalgsWithEtaleAxiom} establishes a connection between them. But first, we prove a couple of auxiliary lemmas. 

\begin{lemma}\label{aux-lemma-1-identity}
If for elements $x,y,z$ in a Heyting algebra $H$ holds $(y\wedge x)=(y\wedge z)$, then $y\leq (x\Leftrightarrow z)$.
\end{lemma}
\begin{proof}
    If $(y\wedge x)=(y\wedge z)$, then $(y\wedge x)=(y\wedge z) \leq z$ and $(y\wedge z)=(y\wedge x)\leq x$. Hence $y\leq (x\Rightarrow z)$ and $y\leq (z\Rightarrow x)$. The last holds if and only if $y\leq (x\Leftrightarrow z)$.
\end{proof}

\begin{lemma}\label{aux-lemma-2-identity}
Let $c\colon H\to A$ be a Heyting $H$-algebra where $H$ is complete. Let $x\in A$ be a fixed element of $A$ and $G\subseteq A$ be a subset of $A$ such that $1_A\in A$ is representable as a join $1_A=\bigvee_{y\in G} y$ (possibly infinite). Suppose also that for arbitrary elements $y\in G$ there exists $h_y\in H$ such that $y\wedge x = y\wedge c(h_y)$, then $\underset{h\in H}{\bigvee} \left(x\Leftrightarrow c(h)\right) = 1$ holds for $c$.
\end{lemma}

\begin{proof}
Fix some element $x\in A$. Using the Lemma \ref{aux-lemma-1-identity}, for an arbitrary element $y\in G$, there exists $h_y\in H$ such that $y\leq (x\Leftrightarrow c(h_y))$. Then $1=\bigvee_{y\in G} y \leq \underset{h_y\in H}{\bigvee} \left(x\Leftrightarrow c(h_y)\right)$.
\end{proof}

\begin{proposition}\label{strictEsakiasAreEquivalentToHalgsWithEtaleAxiom}
Let ${f\colon X'\to X}$ be a $p$-morphism between partially ordered sets $X'$ and $X$. $f$ is a strict $p$-morphism if and only if for every upset $U$ in $X'$, in the Heyting algebra of upsets of $X'$ the following identity holds: \[\underset{W\in \category{Up}(X)}{\bigvee} \!\!\!\!\!\!\!\!\left(U\Leftrightarrow \rightadjunct{f}(W)\right) = X'.\]
\end{proposition}

\begin{proof}
First, suppose that $f$ is a strict $p$-morphism. By the Lemma \ref{aux-lemma-2-identity}, to demonstrate the validity of the identity \textetaleaxiom{H}\!{\!,} it is sufficient to, for any $y\in X'$ and any fixed upset $U$ of $X'$, find an upset $W$ such that $\upsetarrow y\cap U = \upsetarrow y\cap \rightadjunct{f}(W)$.

Let $W=\leftadjunct{f}(\upsetarrow y\cap U)$. Note that $W$ is upset because $f$ is a $p$-morphism. It is also obvious that $\upsetarrow  y\cap U\subseteq \rightadjunct{f}(W)$. Hence $\upsetarrow  y\cap U \subseteq \upsetarrow  y\cap \rightadjunct{f}(W)$. On the other hand, since $f$ is strict, $\upsetarrow y\cap \rightadjunct{f}(W) \subseteq U$, hence $\upsetarrow y\cap \rightadjunct{f}(W) \subseteq \upsetarrow y\cap U$.

Now suppose $f$ validates the identity \textetaleaxiom{H}\!\!. $f$ is a $p$-morphism. To demonstrate the strictness of the $p$-morphism $f$, it is sufficient to show that, for any $y\in X'$ and any elements $y \leqslant y_1, y \leqslant y_2$ above $y$, if $f(y_1) = f(y_2)$, then $y_1 = y_2$.

Suppose, to the contrary, that $f(y_1) = f(y_2)$ and $y_1 \neq y_2$. Without loss of generality, we assume $y_1 \nleqslant y_2$. Since \textetaleaxiom{H} is valid for $f$, for $U = \upsetarrow y_1$, there should exist an upset $W$ in $X$ such that $\upsetarrow y \cap \rightadjunct{f}(W) = \upsetarrow y \cap U = \upsetarrow y_1$. Since $y_2 \notin \upsetarrow y_1$, we also have $y_2 \notin \upsetarrow y \cap \rightadjunct{f}(W)$. Since $y_2 \in \upsetarrow y$, we conclude that $y_2 \notin \rightadjunct{f}(W)$, and herefore, $f(y_2) \notin W$. Recall that we also assumed $f(y_1) = f(y_2)$; hence, $y_1 \notin \rightadjunct{f}(W)$. Therefore, $y_1 \notin \upsetarrow y \cap \rightadjunct{f}(W) = \upsetarrow y_1$, leading to a contradiction.
\end{proof}

The last proposition easily implies the following corollary.
\begin{corollary}\label{Fin_strictEsakiasAreEquivalentToHalgsWithEtaleAxiom}
Let $f\colon X'\to X$ be an Esakia morphism over a finite Esakia space $X$. $f$ is a strict $p$-morphism if and only if $\rightadjunct{f}\colon H_X\to H_Y$ validates the identity \textetaleaxiom{H}\!\!.
\end{corollary}
\begin{proof}
    Indeed, we can safely use the proposition \ref{strictEsakiasAreEquivalentToHalgsWithEtaleAxiom} because a finite Esakia space $X$ is just a finite partially ordered set. 
\end{proof}

Now let us consider the categories $\category{LH}_{\category{E}}/X$ and $\category{S}^{X_H}$. Let $f\colon (X',\leqslant)\to (X,\leqslant)$ be a strict $p$-morphism. Then, for every $x\in X$, we can consider $\rightadjunct{f}(\{x\})$. For every morphism $x_1 \leqslant x_2$ in $X$, according to the 'back' condition for strict $p$-morphisms, one can uniquely define a function $\rightadjunct{f}(\{x_1\}) \to \rightadjunct{f}(\{x_2\})$. This fact suggests that the correspondence just described gives rise to a functor $F_f: (X,\leqslant) \to \category{Set}$. This functor is one part of the pair of functors establishing a connection between $\category{LH}_{\category{E}}/X$ and $\category{S}^{X_H}$ for a finite Esakia space $X_H$.

\begin{proposition}\label{equivalenceOfStricEsakiaAndStonePresheaves}
For a finite Esakia space $X$, the correspondence $f \mapsto F_f$, where $f$ is an object in $\category{LH}_{\category{E}}/X$, gives rise to the equivalence of categories between $\category{LH}_{\category{E}}/X$ and $\category{S}^{X}$.
\end{proposition}

\begin{proof}
Let $f\colon X'\to X$ be an object of $\category{LH}_{\category{E}}/X$. Since $X$ is finite, its Stone topology is discrete. Therefore, for any $x \in X$, the subset $\rightadjunct{f}(\{x\})$ is a clopen subset of $X'$.

Subsets $\rightadjunct{f}(\{x\})$ are Stone spaces. Furthermore, since $f\colon X'\to X$ is continuous, for every $x_1\leqslant x_2$, the induced function $\rightadjunct{f}(\{x_1\})\to \rightadjunct{f}(\{x_2\})$ is continuous. Therefore, $F_f$ is an object in the category $\category{S}^X$.

Moreover, for an arrow $g$ between two $p$-morphisms $f_1\colon X'_1\to X$ and $f_2\colon X'_2\to X$, where $g\colon X'_1\to X'_2$ is a continuous order-preserving function such that $f_2\circ g = f_1$, the induced function $\restr{g}{\rightadjunct{f_1}(\{x\})}\colon \rightadjunct{f_1}(\{x\})\to \rightadjunct{f_2}(\{x\})$ is continuous for any $x\in X$. This establishes a natural transformation $F_{f_1}\to F_{f_2}$, allowing us to conclude that we have constructed a functor $\Phi\colon \category{LH}_{\category{E}}/X\to \category{S}^X$.

To construct a functor $\Psi\colon \category{S}^X\to \category{LH}_{\category{E}}/X$ for an object $F\in \category{S}^X$, we set $\Psi(F)$ to be the canonical projection \[\pi_F\colon \underset{x\in X}{\coprod}F(x)\to X.\] Here, $\underset{x\in X}{\coprod}F(x)$ is the finite topological sum of Stone spaces, with the following partial order relation: For $\xi_1\in F(x_1)$ and $\xi_2\in F(x_2)$, it holds that $\xi_1\leqslant \xi_2$ if and only if $x_1\leqslant x_2$, and the function $F_{xy}$ corresponding to this pair sends $\xi_1$ to $\xi_2$.

Now, it is routine to check that the $\underset{x\in X}{\coprod}F(x)$ is an Esakia space, and the function is a continuous strict $p$-morphism with respect to the topology and partial order relation on $\underset{x\in X}{\coprod}F(x)$.

Note that this functor is constructed using the so-called Grothendieck construction and is often denoted as $\int F$ in the literature; see, e.g., \cite[Ch. 1, \S 5]{maclane-moerdijk-sheaves}.

A natural transformation $\gamma\colon F_1\to F_2$ induces a canonical continuous function \[{\underset{x\in X}{\coprod}\gamma_x\colon \underset{x\in X}{\coprod}F_1(x)\to \underset{x\in X}{\coprod}F_2(x)}.\] It is routine to check that this induces a commutative triangle $\Psi(F_1)\to \Psi(F_2)$. Thus, we have completed the construction of the functor $\Psi\colon \category{S}^X\to \category{LH}_{\category{E}}/X$.

Now, note that for any object $f\colon X'\to X$ in $\category{LH}_{\category{E}}/X$, there exists a homeomorphism $X'\to \underset{x\in X}{\coprod}\rightadjunct{f}(\{x\})$. This gives an isomorphism between $\Psi\circ\Phi$ and the identity functor.

Moreover, for an arbitrary functor $G\colon X\to \category{S}$ and $x\in X$, the space $G(x)$ can be identified with $\rightadjunct{f}(\{x\})$; this gives an isomorphism between $G$ and $\Phi\circ \Psi(G)$.

This last observation completes the proof.
\end{proof}

To complete the picture shown in Figure \ref{fig:outline}, we need to establish a connection between the categories $\etalevariety{H_{X}}$ and $\category{S}^{X_{H}}$, where $X_H$ is a finite partially ordered set and $H_X$ is a finite Heyting algebra, which is dual to $X_H$ via Esakia duality.

\begin{definition}
    A \emph{subfunctor} of a functor $\functor{G}\colon \category{D}\to \category{C}$ between categories $\category{D}$ and $\category{C}$ is a pair $(\functor{F}, i)$ where $\functor{F}\colon \category{D}\to \category{C}$ is a functor and $i\colon \functor{F}\to \functor{G}$ is a natural transformation such that its components $i_A\colon \functor{F}(A)\to \functor{G}(A)$ are monomorphisms. Sometimes subfunctors are also called \emph{subpresheaves}.
\end{definition}

\begin{definition}
    The \emph{image} of a morphism $f\colon A\to B$ in a category $\category{C}$ is a universal factorization of $f$ into the composite $A\to \operatorname{Im}(A)\to B$ of an epimorphism and a monomorphism, so that $\operatorname{Im}(A)$ is a subobject of $B$.
\end{definition}

In a case like ours, the notion of an image is well-defined. Specifically, the image of a morphism $f$ is precisely the regular, set-theoretic image of a function. In this context, a subfunctor simply means that for any object $A$ in the category $\category{C}$, the object $\functor{F}(A)$ is a subobject of $G(A)$; and for a morphism $f\colon A\to B$, the corresponding morphism $\functor{F}(f)$ is the restriction of $\functor{G}(f)$ to $\functor{F}(A)$.

\begin{proposition}\label{StonePresheavesOnSierpinskiToEtaleVariety}
For a finite Heyting algebra $H$, the categories $\category{S}^{X_{H}}$ and $\etalevariety{H_X}$ are dually equivalent.
\end{proposition}

\begin{proof}
The objects of $\category{S}^{X_H}$ are covariant functors $F\colon X_H\to \category{S}$.
Hence $F(X_H)$ is a diagram of Stone spaces with $F_{xy}\colon F(x)\to F(y)$ for each $x\leqslant y$.

Using the Grothendieck construction, we construct a strict $p$-morphism for each functor $F$, the canonical projection $\pi\colon \int F\to X_H$.

In fact, $\int F$ is an Esakia space with the underlying set $\coprod_{x\in X_H} F(x)$, equipped with the induced topology of coproduct and the partial order relation, defined as $\xi \leqslant \restr{\xi}{xy}$ for each $x\in X_H$, $\xi\in F(x)$, where $\restr{\xi}{xy}=F_{xy}(\xi)\in F(y)$ for $x\leqslant y$; see the Proposition \ref{equivalenceOfStricEsakiaAndStonePresheaves}.

Our goal is to show that for a finite poset $X_H$, the Esakia dual of the morphism $\pi\colon \int F\to X$ belongs to $\etalevariety{H_X}$.

Let $H(F)$ be the Esakia dual of $\int F$. We claim that $H(F)$ is isomorphic to a subalgebra of $\prod_{\begin{subarray}{l}x\in X \\ \xi\in F(x)\end{subarray}}\!\!{\category{Up}(\upsetarrow x)}$, where $\category{Up}(\upsetarrow x)$ denotes the algebra of upsets in $\upsetarrow x$.

The elements of $H(F)$ are clopen upsets of $\int F$; these are the clopen subfunctors of $F$. In other words, they are functors $F'$ such that $F'(x) \subseteq F(x)$ for each $x \in X$.

It is important to note that subfunctoriality on arrows is necessary to ensure that the clopen subset is upward-closed. Additionally, both $F'(x)$ and $F(X)$ must be clopens in $\int F$ (in our case, given that $X$ is finite, $F(x)$ are indeed clopens).

Let, for $x\in X$ and $\xi\in F(x)$, a function $m_{x,\xi}\colon H(F)\to \category{Up}(\upsetarrow x)$ be defined as follows: for a subfunctor $F'\subseteq F$, let $F'\mapsto \{y \mid x\leqslant y \text{ and } \restr{\xi}{xy}\in F'(y)\}$. Note that if $y \in m_{x,\xi}(F')$ and $y \leqslant z$, then $x \leqslant z$ and $\restr{\xi}{xz} = \restr{(\restr{\xi}{xy})}{yz} \in F'(z)$, so $m_{x,\xi}(F')$ is an upset.

We claim that for each $x\in X$ and $\xi\in F(x)$, the function $m_{x,\xi}$ is a Heyting algebra homomorphism. Knowing this gives us the desired inclusion ${H(F)\to \prod_{\begin{subarray}{l}x\in X \ \xi\in F(x)\end{subarray}}{\category{Up}(\upsetarrow x)}}$.

For two clopen subfunctors of $F$, $F'_1$ and $F'_2$, the meet $F'_1\wedge F'_2$ at $x\in X$ is defined by $F'_1(x)\cap F'_2(x)$ and $m_{x,\xi}(F'_1\wedge F'_2)=\{y\mid x\leqslant y \text{ and } \restr{\xi}{xy}\in (F'_1\wedge F'_2)(y)\}$.

In a similar vein, $F'_1(x)\cup F'_2(x)$ defines the join $F'_1 \vee F'_2$ at $x\in X$, and $m_{x,\xi}(F'_1\vee F'_2)=\{y\mid x\leqslant y \text{ and } \restr{\xi}{xy}\in (F'_1\vee F'_2)(y)\}$.

Straightforward computations reveal that $m_{x,\xi}$ preserves both the meet and join operations.

For two clopen subfunctors $F'_1$ and $F'_2$, the implication $(F_1\Rightarrow F_2)(x)$ at $x\in X$ should be defined as the greatest clopen subfunctor $F'_3$ of $F$, such that $F'_3 \wedge F'_1$ is a subfunctor of $F'_2$.

By definition, the implication $U_1\Rightarrow U_2$ in $\category{Up}\left(\upsetarrow x\right)$ is the greatest upward-closed subset of $\left(\upsetarrow x\setminus U_1\right)\cup U_2$. Therefore, the implication $m_{x,\xi}(F'_1(x))\Rightarrow m_{x,\xi}(F'_2(x))$ is equal to the greatest upward-closed subset of \[\{y\mid x\leqslant y \text{ and } \restr{\xi}{xy}\not \in F'_1(y)\} \cup \{y\mid x\leqslant y \text{ and } \restr{\xi}{xy}\in F'_2(y)\}.\]

At this point, we should note that proving the direct preservation of the implication seems challenging. However, we will see shortly that using Esakia duality simplifies our task.

To prove that $m_{x,\xi}\colon H(F)\to \category{Up}(\upsetarrow x)$ is a Heyting algebra homomorphism, we look at the Esakia dual function $\left(m_{x,\xi}\right)_{\category{E}}$ that connects the Esakia duals of the codomain and the domain of $m_{x,\xi}$. We need to prove that the function $\left(m_{x,\xi}\right)_{\category{E}}\colon \upsetarrow x\to\int F$ is a continuous $p$-morphism.

Remember that $\pi\colon \int F\to X$ is a strict $p$-morphism and that $\upsetarrow \xi$ and $\upsetarrow\pi(\xi)=\upsetarrow x$ are the smallest open neighborhoods of $\xi$ and $x$ in the upset topology of $\int F$ and $X$. Hence, $\pi\colon \upsetarrow \xi \to \upsetarrow \pi(\xi)$ is an isomorphism.

On the other hand, to put it simply, $\left(m_{x,\xi}\right)_{\category{E}}\colon \upsetarrow x\to\int F$ is equal to the following composition: \[\upsetarrow x\xrightarrow{\pi^{-1}} \upsetarrow \xi\xhookrightarrow{\upsetarrow \xi\cap F'(-)} \int F.\]

Firstly, $\upsetarrow x$ is discrete, and $\left(m_{x,\xi}\right)_{\category{E}}$ is continuous. Secondly, suppose $y \leqslant z$ in $\upsetarrow x$, then $\restr{\xi}{xz}=\restr{(\restr{\xi}{xy})}{yz}$, hence $\restr{\xi}{xy}\leqslant \restr{\xi}{xz}$, and $\left(m_{x,\xi}\right)_{\category{E}}$ is order-preserving. Lastly, by the definition of the order relation on $\int F$, we know that $\restr{\xi}{xy}\leqslant \restr{\xi}{xz}$ holds if and only if $y\leqslant z$. Therefore, $\left(m_{x,\xi}\right)_{\category{E}}$ is a $p$-morphism.

As a final step, to demonstrate that the dual of $\pi$ is in the variety $\etalevariety{H_X}$, we need to verify that for arbitrary elements $x\in X$, the algebra $\category{Up}\left(\upsetarrow x\right)$ belongs to $\etalevariety{H_X}$. Note that the latter is indeed the case because, for any $x \in X$, the inclusion $\upsetarrow x \hookrightarrow X$ is a continuous $p$-morphism, and hence $\category{Up}(\upsetarrow x)$ is a homomorphic image of $H_X$.
\end{proof}

We conclude this section by demonstrating some applications of the results we have proven.

\begin{corollary}\label{pushout-corrolary}
Let $c_i:H\to H_i$, $i=1,2$ be objects of the variety $\etalevariety{H}$, for a finite Heyting algebra $H$. Then their coproduct in $\etalevariety{H}$ is isomorphic to the pushout of $c_1$ and $c_2$ in the category of distributive lattices.     
\end{corollary}
\begin{proof}
Let us denote the coproduct of $c_1$ and $c_2$ in $\etalevariety{H}$ by $c_1\tilde\otimes c_2\colon H\to H_1\underset{H}{\tilde\otimes} H_2$.

Let $X$ be the Esakia dual of $H$. Let $f_i\colon X_i\to X$, $i=1,2$, represent the objects of $\category{LH}_{\category{E}}/X$ corresponding to $c_1$ and $c_2$, and let $f_1\underset{X}{\tilde\times}f_2$ denote their product in $\category{LH}_{\category{E}}/X$, dual to $c_1\tilde\otimes c_2$ above.

According to Proposition \ref{equivalenceOfStricEsakiaAndStonePresheaves}, let $F_1$ and $F_2$ be functors, i.e., objects of $\category{S}^X$, corresponding to $f_1$ and $f_2$.

Note that in functor categories, products are computed pointwise, namely $(F_1 \times F_2)(x) = F_1(x) \times F_2(x)$; see, e.g., \cite[Ch 1, \S2]{maclane-moerdijk-sheaves}.

Using this, it is straightforward to check that $F_1\times F_2$ corresponds to the pullback $f_1\underset{X}{\times}f_2$ as computed in the category of ordered spaces and arbitrary order-preserving continuous maps.

That is, the underlying set of $f_1\underset{X}{\times}f_2$ is the pullback $X_1\underset{X}\times X_2$ of underlying sets of $f_1$ and $f_2$ in the category of sets. Thus, $f_1\underset{X}{\tilde\times}f_2$ and $f_1\underset{X}{\times}f_2$ are isomorphic.

By Priestley duality, we know that the pullback $X_1\underset{X}{\times}X_2$ is dual to the pushout $H_1\underset{H}{\otimes}H_2$ of the corresponding distributive lattices. Thus, the algebra $H_1\underset{H}{\tilde\otimes} H_2$ in the $\etalevariety{H}$-coproduct $c_1\tilde\otimes c_2\colon H\to H_1\underset{H}{\tilde\otimes} H_2$ is isomorphic to the pushout $H_1\underset{H}{\otimes}H_2$ in distributive lattices.
\end{proof}

Moreover, the following holds: 

\begin{lemma}\label{finite-limits-lemma}
For a finite partially ordered set $X$, finite limits in $\category{LH}_{\category{E}}/X$ coincide with ones computed in ordered spaces and continuous order-preserving maps.
\end{lemma}
\begin{proof}
It is clear that the terminal object $id\colon X\to X$ in $\category{LH}_{\category{E}}/X$ coincides with the terminal object in ordered spaces over $X$.

Now for a tripple $f_i\colon Y_i\to X$, where $i\in \{1,2,3\}$, let us consider a diagram $f_2\overset{\!\alpha}{\rightarrow} f_1 \overset{~\beta}{\leftarrow} f_3$ in $\category{LH}_{\category{E}}/X$.

Then the pullback $f_2\underset{f_1}{\times}f_3$ is equal to $f_1\circ \pi$, where $\pi$ is the product $\alpha\times \beta$ (which is given by the pullback $P=Y_2\underset{Y_1}{\times}Y_3$ over $Y_1$).

\begin{figure}[h]
    \centering
    \begin{tikzpicture}[ampersand replacement=\&,>=stealth,->]
\matrix[matrix of math nodes, column sep = {1.5cm,between origins}, row sep = 1cm]
{
|(P)|P\&|(E)|Y_2\\
|(X)|Y_3\&|(F)|Y_1\\
};
\begin{scope}[every node/.style={midway,auto,font=\scriptsize}]
\draw (P) --  (E);
\draw (P) --  (X);
\draw[->] (X) -- node[below]{$\beta$} (F);
\draw (E) -- node[right]{$\alpha$} (F);

\end{scope}
\end{tikzpicture}
    \caption{Pullback of $\alpha, \beta$}
    \label{fig:enter-label}
\end{figure}
    
It has already been shown in the proof of Corollary \ref{pushout-corrolary} that products in $\category{LH}_{\category{E}}/X$ are the same as products found in the category of ordered topological spaces and continuous order-preserving functions. As a result, pulbacks in $\category{LH}_{\category{E}}/X$ coincide with those computed in the category of ordered topological spaces and continuous order-preserving functions.

Since, according to \cite[Proposition 2.8.2]{borceux1}, finite limits are constructible by the terminal object and pullbacks, we get that all finite limits in $\category{LH}_{\category{E}}/X$ coincide with ones computed in the category of ordered topological spaces and continuous order-preserving functions.
\end{proof}

\begin{corollary}\label{finite-colimits-cor}
    For a finite Heyting algebra $H$, all finite colimits in $\etalevariety{H}$ coincide with ones computed in the category of distributive lattices.
\end{corollary}
\begin{proof}
        Use Lemma \ref{finite-limits-lemma} and the duality between $\etalevariety{H}$ and $\category{LH}_{\category{E}}/X$. 
\end{proof}

\section{Conclusion and future work}
\subsection{General case}
A natural follow-up question is whether the categories $\etalevariety{H}$ and $\category{LH}_{\category{E}}/X$ are still dually equivalent in the case of arbitrary, infinite Heyting algebra $H$. Unfortunately, the answer to this question appears to be negative.

Nonetheless, first note the following:

\begin{proposition}
    For an arbitrary  continuous strict $p$-morphism $f\colon X'\to X$ the Esakia dual of $f$ belongs to $\etalevariety{H_X}$.
\end{proposition}
\begin{proof}
Similar to the approach used in Proposition \ref{StonePresheavesOnSierpinskiToEtaleVariety}, we can demonstrate that $H_X'$ is isomorphic to a subalgebra of $\prod_{ \begin{subarray}{l}x\in X \\ \xi\in \rightadjunct{f}(\{x\})\end{subarray}}{H_{\upsetarrow x}}$.

Let $X$ be an arbitrary Esakia space and $f\colon X'\to X$ a strict Esakia morphism. The elements of $X'$ can be indexed by pairs $(x,\xi)$ where $x\in X$ and $\xi\in \rightadjunct{f}(\{x\})$. Since $f$ is a strict $p$-morphism, then $\restr{f}{\upsetarrow \xi}$ is a continuous bijection between $\upsetarrow \xi$ and $\upsetarrow f(\xi)$.

Recall that $\upsetarrow \xi$ is the smallest upset containing $\xi$, which, by Lemma \ref{esakia_space_pointed_upset_is_closed}, is always closed in an Esakia space. A closed subspace of a Hausdorff space is compact Hausdorff as well. Hence, $\restr{f}{\upsetarrow \xi}$ is an isomorphism because it is a continuous bijection of compact Hausdorff spaces.

The map $\left(m_{x,\xi}\right)_{\category{E}}\!\colon \uparrow x\xrightarrow{f^{-1}} \uparrow \xi\xhookrightarrow{\uparrow \xi\cap \rightadjunct{f}({x})} X'$ is clearly a continuous map.

Let's suppose that $y\leqslant z$ in $\uparrow x$. Then, the map $\left(m_{x,\xi}\right)_{\category{E}}$ is an order-preserving map due to the fact that $f$ is a $p$-morphism. Given that $f$ is order-preserving, it follows that $\left(m_{x,\xi}\right)_{\category{E}}$ is a $p$-morphism.

Once again, to complete the proof that the Esakia dual of $f$ is in the variety $\etalevariety{H_X}$, we need to verify that for arbitrary elements $x\in X$, the algebra $H_{\upsetarrow x}$ belongs to $\etalevariety{H_X}$. This is trivially the case since, for any $x\in X$, the inclusion $\upsetarrow x\hookrightarrow X$ is a continuous $p$-morphism, and hence $H_{\upsetarrow x}$ is a homomorphic image of $H_X$.  
\end{proof}

The inverse direction makes things harder. Namely, for infinite Esakia spaces $X$, there exist non-strict p-morphisms $f\colon X'\to X$ whose Esakia duals do belong to $\etalevariety{H_X}$.

\begin{example}
Consider the ordinals $X=\omega+1$ and $X'=\omega+2$, along with their natural orderings and ordinal topologies. It is straightforward to verify (and well known) that both are Esakia spaces. Consider the map $\pi\colon X'\to X$, sending all elements of $\omega+1$ to themselves and the element $\omega+1\in X'$ to $\omega\in X$. It is evident that $\pi$ is a continuous order-preserving $p$-morphism that is not strict. Indeed, under a strict $p$-morphism, inverse images of points are incomparable, while $\pi(\omega)=\pi(\omega+1)=\omega$ with $\omega<\omega+1$.

On the other hand, the dual homomorphism $\esakiadualhomo{\pi}\colon H_X\to H_{X'}$ belongs to $\etalevariety{H_X}$.

To see this, we will explicitly represent $H_{X'}$ as a subquotient of a product of several copies of $H_X$, in a way compatible with $\esakiadualhomo{\pi}$ and the diagonal embedding of $H_X$ into that product.

For that, it will be convenient for us to identify $H_X$ with the set of all natural numbers with reverse ordering, together with an extra bottom element,
\[
H=\{1\succ 2 \succ 3 \succ \cdots \succ \bot\}.
\]
Here $n\in H$, $n=1,2,3,\ldots$ corresponds to the principal clopen upset $\upsetarrow n$ of $X$, and $\bot$ to the empty clopen upset.

Similarly, $H_{X'}$ can be identified with $H'=H\cup\{\bot'\}$, where $\bot'$ is an element with $h\succ\bot'$ for all $h\in H$ distinct from $\bot$ and $\bot'\succ\bot$. This $\bot'$ corresponds to the clopen upset $\{\omega+1\}$ of $X'$. 

Under this identification, $\esakiadualhomo{\pi}$ corresponds to the identical embedding of $H$ into $H'$. So let us show that the $H$-algebra $H'$ defined by $\esakiadualhomo{\pi}$ indeed belongs to $\etalevariety{H}$. Namely, it is a quotient of a subalgebra of the product of an infinite number of copies of $H$.

Let $N\subset H$ be $H\setminus\{\bot\}$. Let $H^N$ be the product of $N$ many copies of $H$, which we will view as the algebra of all maps from $N$ to $H$ with pointwise Heyting algebra structure. Let $A\subset H^N$ be the Heyting subalgebra of $H^N$ generated by the identical embedding $\boldsymbol x\colon N\hookrightarrow H$, together with all constant maps $\boldsymbol c_h$, $h\in H$, with $\boldsymbol c_h(n)= h$ for all $n\in N$. Note also that $\neg\boldsymbol x=\neg\boldsymbol c_h=\boldsymbol c_\bot$.

Since the image of the diagonal embedding $d\colon H\hookrightarrow H^N$ consists precisely of the constant maps $\boldsymbol c_h$ above, $A$ contains the image of $d$, so $a\colon H\hookrightarrow A$ belongs to $\etalevariety{H}$.

Let $\mathcal F\subseteq A$ be the subset of $A$ consisting of those $f\in A$ which satisfy $f(n)=1$ for almost all $n$, i.e., there exists an $n_0\in N$ such that $f(n)=1$ for all $n\prec n_0$. This $\mathcal F$ is a filter of $A$.

Note that,
\[
(\boldsymbol x\mathimplies \boldsymbol c_h)(n)=
\begin{cases}
h, \text{ if } n\succ h\\
1, \text{ otherwise}
\end{cases}
\]
which belongs to $\mathcal F$ for all $h\neq \bot$ and is equal to $\boldsymbol c_\bot$ for $h=\bot$.

Moreover
\[
\boldsymbol x \cong_{\mathcal F} (\boldsymbol c_h\mathimplies \boldsymbol x)(n)=
\begin{cases}
1, \text{ if } n\succ h \text{ or } n=h\\
n, \text{ otherwise}
\end{cases}; 
\]

\[
 \boldsymbol c_h \cong_{\mathcal F} (\boldsymbol c_h\vee \boldsymbol x)(n)=
\begin{cases}
n, \text{ if } n\succ h \text{ or } n=h\\
h, \text{ otherwise}
\end{cases}; 
\]

\[
\boldsymbol x \cong_{\mathcal F} (\boldsymbol c_h\wedge \boldsymbol x)(n)=
\begin{cases}
h, \text{ if } n\succ h \text{ or } n=h\\
n, \text{ otherwise}
\end{cases}.
\]

It follows that in the quotient algebra $A/\mathcal F$ we have $[\boldsymbol x]\prec [\boldsymbol c_n]$ for all $h\neq \bot$ and $[\boldsymbol x]\succ [\boldsymbol c_\bot]=[\bot]$. Also, $[\boldsymbol c_h]\succ [\boldsymbol c_{h+1}]$ for all $h\neq \bot$ as $\boldsymbol c_h\mathimplies \boldsymbol c_{h+1}\notin\mathcal F$, and $[\boldsymbol x]\ne[\bot]$ since $\boldsymbol x\mathimplies\bot\notin\mathcal F$. Also note that from the computations shown above, it follows that $A/\mathcal F$ does not contain any other elements because $A$ is generated by the $\boldsymbol c_h$ and by $\boldsymbol x$.

We therefore conclude that $A/\mathcal F$ is isomorphic to $H'$. Moreover, this isomorphism is compatible with the homomorphisms $d\colon H\hookrightarrow A$ and $\esakiadualhomo{\pi}\colon H\to H'$, so that indeed $H'$ belongs to $\etalevariety{H}$.
\end{example}

\section*{Acknowledgments}
I am grateful to Mamuka Jibladze and David Gabelaia, who have generously contributed to this research effort with their interest and thoughtful suggestions.

\section*{Declarations}
\textbf{Ethical Approval: } Not applicable.\\
\textbf{Funding :} This work was supported by the Shota Rustaveli National Science Foundation of Georgia (SRNSFG) grant \#PHDF-23-109.\\
\textbf{Financial or Non-financial interests :} The author has no competing interests to declare that are relevant to the content of this article.


\end{document}